\documentclass[12pt]{amsart}
\date{}
\newtheorem{theorem}{Theorem}
\newtheorem{lemma}{Lemma}
\newtheorem{corollary}{Corollary}
\newtheorem{proposition}{Proposition}

\newtheorem{problem}{Problem}

\newtheorem{example}{Example}
\theoremstyle{definition}\newtheorem{remark}{Remark}

\newcommand{\Aff}{\mathrm{Aff}}
\newcommand{\osc}{\mathrm{osc}}

\newcommand{\IN}{{\mathbb{N}}}
\newcommand{\IQ}{{\mathbb{Q}}}
\newcommand{\IZ}{{\mathbb{Z}}}
\newcommand{\IR}{{\mathbb{R}}}
\newcommand{\IC}{{\mathbb{C}}}
\newcommand{\IT}{{\mathbb{T}}}
\newcommand{\Arg}{{\mathrm{Arg}}}
\newcommand{\Ker}{{\mathrm{Ker}}}

\begin{document}
\title[Oscillator topologies on a paratopological group]
{Oscillator topologies on a paratopological group and related
number invariants}
\author{Taras Banakh}
\author{Olexandr Ravsky}
\email{tbanakh@franko.lviv.ua; oravsky@mail.ru}
\address{Department of Mathematics, Ivan Franko Lviv National University,
Universytetska 1, Lviv, 79000, Ukraine}
\begin{abstract}
We introduce and study oscillator topologies on paratopological
groups and define certain related number invariants. As an
application we prove that a Hausdorff paratopological group $G$
admits a weaker Hausdorff group topology provided $G$ is
3-oscillating. A paratopological group $G$ is 3-oscillating (resp.
2-oscillating) provided for any neighborhood $U$ of the unity $e$
of $G$ there is a neighborhood $V\subset G$ of $e$ such that
$V^{-1}VV^{-1}\subset UU^{-1}U$ (resp. $V^{-1}V\subset UU^{-1}$).
The class of 2-oscillating paratopological groups includes all
collapsing, all nilpotent paratopological groups, all
paratopological groups satisfying a positive law, all
paratopological SIN-group and all saturated paratopological groups
(the latter means that for any nonempty open set $U\subset G$ the
set $U^{-1}$ has nonempty interior). We prove that each totally
bounded paratopological group $G$ is countably cellular; moreover,
every cardinal of uncountable cofinality is a precaliber of $G$.
Also we give an example of a saturated paratopological group which
is not isomorphic to its mirror paratopological group as well as
an example of a 2-oscillating paratopological group whose mirror
paratopological group is not 2-oscillating.
\end{abstract}

\maketitle \baselineskip15pt

This note was motivated by the following question of I.~Guran
\cite{Gu}: {\em Does every (regular) Hausdorff paratopological
group $G$ admit a weaker Hausdorff group topology?}

Under a {\em paratopological group\/} we understand a pair
$(G,\tau)$ consisting of a group $G$ and a topology $\tau$ on $G$
making the group operation $\cdot:G\times G\to G$ of $G$
continuous. If, in addition, the operation $(\cdot)^{-1}:G\to G$
of taking the inverse is continuous with respect to the topology
$\tau$, then $(G,\tau)$ is a {\em topological group\/}. A
paratopological group $G$ is {\em Lawson\/} if $G$ possesses a
neighborhood base at the unit, consisting of subsemigroups of $G$.
Under the {\em mirror paratopological group\/} of a
paratopological group $G=(G,\tau)$ we understand the
paratopological group $G^-=(G,\tau^{-1})$ where
$\tau^{-1}=\{U^{-1}:U\in\tau\}$. Let us mention that there are
paratopological groups which are not isomorphic to their mirror
paratopological groups, see Examples~\ref{ex11} and \ref{ex3}.

Given a paratopological group $G$ let $\tau_\flat$ be the
strongest group topology on $G$, weaker than the topology of $G$.
The topological group $G^\flat=(G,\tau_\flat$), called {\em the
group reflexion\/} of $G$, has the following characteristic
property: the identity map $i:G\to G^\flat$ is continuous and for
every continuous group homomorphism $h:G\to H$ from $G$ into a
topological group $H$ the homomorphism $h\circ i^{-1}:G^\flat\to
H$ is continuous. Our definition of the topology $\tau_\flat$ is
categorial. An inner description of the topology $\tau_\flat$ can
be given using the technique of $T$-filters, see \cite[\S3.1]{PZ}.
A subset $A$ of a paratopological group $G$ will be called {\em
$\flat$-closed\/} (resp. {\em $\flat$-open\/}) if $A$ is closed
(resp. open) in the topology $\tau_\flat$. A paratopological group
$G$ is called {\em $\flat$-separated\/} provided its group
reflexion $G^\flat$ is Hausdorff. We define a paratopological
group $G$ to be {\em $\flat$-regular\/} if each neighborhood $U$
of the unit $e$ of $G$ contains a $\flat$-closed neighborhood of
$e$. Observe that each Hausdorff\/ $\flat$-regular paratopological
group is regular and\/ $\flat$-separated. The latter assertion
follows from the fact that every point $x\not=e$ of $G^\flat$ can
be separated from the unit $e$ by a $\flat$-closed subset. This
implies that the topological group $G^\flat$ is separated and
hence Hausdorff.

Observe that in terms of group reflexions the Guran question can
be reformulated as follows: {\em Is any (regular) Hausdorff
paratopological group $G$ $\flat$-separated?}

The negative answer to this question was given by the second
author in \cite{Ra1} where he has constructed a non-commutative
Hausdorff zero-dimensional paratopological group with
non-Hausdorff group reflexion. In fact, any such a paratopological
group necessarily is non-commutative: according to \cite{Ra1} the
group reflexion $G^\flat$ of any abelian Hausdorff paratopological
group $G$ is Hausdorff. Moreover, in this case the topology of
$G^\flat$ has a very simple description: a base of neighborhoods
at the unit in $G^\flat$ consists of the sets $UU^{-1}$ where $U$
runs over neighborhoods of the unit in the group $G$.  A bit later
it was realized that the same is true for any {\em paratopological
SIN-group\/}, that is a paratopological group $G$ possessing a
neighborhood base $\mathcal B$ at the unit such that $gUg^{-1}=U$
for any $U\in\mathcal B$ and $g\in G$ (as expected, SIN is
abbreviated from {\bf S}mall {\bf I}nvariant {\bf N}eighborhoods).
Unfortunately, Hausdorff paratopological SIN-groups do not exhaust
all paratopological groups whose group reflexion is Hausdorff (for
example any separated topological group has Hausdorff group
reflexion but needs not be a paratopological SIN-group). In this
situation it is natural to search for less restrictive conditions
on a paratopological group $G$ under which the group reflexion
$G^\flat$ of $G$ is Hausdorff and admits a simple description of
its topology. This is important since many results concerning
paratopological groups require their $\flat$-separatedness, see
\cite{KRS}, \cite{BR1}--\cite{BR5}.

For each paratopological group $(G,\tau)$ we define a decreasing
sequence $\tau=\tau_1\supset\tau_2\supset\cdots\supset\tau_\flat$
of so-called oscillator topologies which are intermediate between
the topology $\tau$ of $G$ and the topology $\tau_\flat$ of its
group reflexion. In some fortunate cases the topology $\tau_\flat$
coincides with some oscillator topology $\tau_n$ and thus admits a
relatively simple description.

Given a subset $U$ of a group $G$ by induction define the sets
$(\pm U)^n$ and $(\mp U)^n$, $n\in\omega$, letting $(\pm U)^0=(\mp
U)^0=\{e\}$ and $(\pm U)^{n+1}=U(\mp U)^n$, $(\mp
U)^{n+1}=U^{-1}(\pm U)^n$ for $n\ge 0$. Thus $(\pm
U)^n=\underset{n}{\underbrace{UU^{-1}U\cdots U^{(-1)^{n-1}}}}$ \
and $(\mp U)^n=\underset{n}{\underbrace{U^{-1}UU^{-1}\cdots
U^{(-1)^{n}}}}$. Note that $((\pm U)^n)^{-1}=(\pm U)^n$ if $n$ is
even and $((\pm U)^n)^{-1}=(\mp U)^n$ if $n$ is odd.

Under an {\em $n$-oscillator\/} (resp. {\em a mirror
$n$-oscillator\/}) on a topological group $(G,\tau)$ we understand
a set of the form $(\pm U)^n$ (resp. $(\mp U)^n$~) for some
neighborhood $U$ of the unit of $G$. Observe that each
$n$-oscillator in a paratopological group $(G,\tau)$ is a mirror
$n$-oscillator in the mirror paratopological group $(G,\tau^{-1})$
and vice versa: each mirror $n$-oscillator in $(G,\tau)$ is an
$n$-oscillator in $(G,\tau^{-1})$.

Under the {\em $n$-oscillator topology\/} on a paratopological
group $(G,\tau)$ we understand the topology $\tau_n$ consisting of
sets $U\subset G$ such that for each $x\in U$ there is an
$n$-oscillator $(\pm V)^n$ with $x\cdot(\pm V)^n\subset U$.

Since $(\pm V)^{n+1}\supset (\pm V)^n\cup(\mp V)^n$ for each set
$V$ containing the unit of $G$, we get $\tau_{n+1}\subset\tau_n$
and $\tau_{n+1}\subset(\tau^{-1})_n$ for every $n\in\IN$. Thus we
obtain a decreasing sequence $\tau=\tau_1\supset
\tau_2\supset\cdots\supset\tau_\flat$ of oscillator topologies on
the paratopological group $(G,\tau)$ and also a decreasing
sequence
$\tau^{-1}=(\tau^{-1})_1\supset(\tau^{-1})_2\supset\dots\supset\tau^{-1}_\flat=\tau_\flat$
of oscillator topologies on the mirror paratopological group
$(G,\tau^{-1})$. Observe that $(\tau_n)^{-1}=\tau_n$ if $n$ is
even and $(\tau_n)^{-1}=(\tau^{-1})_n$ if $n$ is odd.

In general, $(G,\tau_n)$ is not a paratopological group but it is
a {\em semitopological group\/}, that is a group endowed with a
topology making the group operation separately continuous
(equivalently, a group endowed with a shift-invariant topology).
The following theorem detects the situation when the sequence of
oscillator topologies eventually stabilize.

\begin{theorem}\label{base} For a paratopological group $(G,\tau)$
and a positive integer $n$  the following conditions are
equivalent:
\begin{enumerate}
\item $(G,\tau_n)$ is a topological group;
\item $\tau_n=\tau_\flat$;
\item $\tau_k=\tau_\flat=(\tau^{-1})_{k+1}$ for any $k\ge n$;
\item $\tau_n\subset(\tau^{-1})_n$ which means that for any
neighborhood $U$ of the unit $e$ of $G$ there is a neighborhood
$V\subset G$ of $e$ such that $(\mp V)^n\subset(\pm U)^n$.

\noindent Moreover, if $n$ is even, then the conditions (1)--(4)
are equivalent to
\item $(G,\tau_n)$ is a paratopological group.
\end{enumerate}
\end{theorem}

\begin{proof} If $n$ is even, then $(1)\Leftrightarrow (5)$
 because of the equality $(\tau_n)^{-1}=\tau_n$.

The implication $(1)\Rightarrow (2)$ follows from the inclusions
$\tau\supset\tau_n\supset\tau_\flat$ and the fact that
$\tau_\flat$ is the strongest group topology weaker than $\tau$.

The implication $(2)\Rightarrow (3)$ follows from the inclusions
$\tau_\flat\subset (\tau^{-1})_{k+1}\subset
(\tau^{-1})_{n+1}\subset \tau_n\supset\tau_k\supset\tau_\flat$
holding for each $k\ge n$.

The implication $(3)\Rightarrow (4)$ follows from the inclusion
$\tau_\flat\subset(\tau^{-1})_n$.

Finally, we show that $(4)\Rightarrow (1)$. Let $\mathcal N(e)$ be
a base of open neighborhoods at the unit $e$ of the
paratopological group $G$. Assume that
$\tau_n\subset(\tau^{-1})_n$ which means that for any
$U\in\mathcal N(e)$ there is $V\in\mathcal N(e)$ with $(\mp
V)^n\subset(\pm U)^n$.

To show that $(G,\tau_n)$ is a topological group we shall use the
Pontriagin characterization \cite[\S18]{Po} asserting that a group
$G$ endowed with a shift-invariant topology is a topological group
if and only if the family $\mathcal B$ of open neighborhoods of
the unit $e$ of $G$ satisfies the following five Pontriagin
conditions:

(P1) $(\forall U,V\in\mathcal B)(\exists W\in\mathcal B)$ with
$W\subset U\cap V$;

(P2) $(\forall U\in\mathcal B)(\exists V\in\mathcal B)$ with
$V^2\subset U$;

(P3) $(\forall U\in\mathcal B)(\forall x\in U)(\exists
V\in\mathcal B)$ with $xV\subset U$;

(P4) $(\forall U\in\mathcal B)(\forall x\in G)(\exists
V\in\mathcal B)$ with $x^{-1}Vx\subset U$;

(P5) $(\forall U\in\mathcal B)(\exists V\in\mathcal B)$ with
$V^{-1}\subset U$.

Thus to prove that $(G,\tau_n)$ is a topological group, it
suffices to verify the Pontriagin conditions (P1)--(P5) for the
family $\mathcal B$ of all $n$-oscillators in $G$.

The first condition (P1) is trivial.

To verify (P2), fix any open neighborhood $U\in\mathcal N(e)$ and
by finite induction find open neighborhoods $U_0\supset
U_1\supset\cdots\supset U_n$ of $e$ in $G$ such that $U_0=U$,
$U_k\cdot U_k\subset U_{k-1}$ if $k$ is odd and $(\mp
U_k)^n\subset(\pm U_{k-1})^n$ if $k$ is even. It is easy to see
that $(\pm U_n)^n\cdot(\pm U_n)^n\subset(\pm U_0)^n=(\pm U)^n$ and
thus the Pontriagin condition (P2) is satisfied too.

(P3) Fix any neighborhood $U\in\mathcal N(e)$ and a point $x=(\pm
U)^n$. We have to find $V\in\mathcal N(e)$ such that $x(\pm
V)^n\subset(\pm U)^n$. Write $x=x_1x_2^{-1}x_3\cdots
x_n^{(-1)^{n-1}}$, where all $x_i$ are in $U$. By $A$ denote the
(finite) set of all products in the forms $y_1\cdots y_n$ where
$y_i\in \{e,x_1,x_2,\dots x_n,x_1^{-1},x_2^{-1},\dots, x_n^{-1}\}$
for every $i$. Choose a neighborhood $V\in\mathcal N(e)$ such that
$(xa^{-1}Va)\cup (a^{-1}Vax) \subset U$ for every
$x\in\{x_1,\dots,x_n\}$ and every $a\in A$. Then $$x(\pm
V)^n=x_1x_2^{-1}x_3\cdots x_n^{(-1)^{n-1}}(\pm
V)^n=x_1(x_2^{-1}x_3\cdots x_n^{(-1)^{n-1}}Vx_n^{(-1)^n}\cdots
x_3^{-1}x_2)\times$$ $$\times x_2^{-1}(x_3\cdots
x_n^{(-1)^{n-1}}V^{-1}x_n^{(-1)^n}\cdots x_3^{-1})\cdots
x_n^{(-1)^{n-1}}V^{(-1)^{n-1}} \subset UU^{-1}\cdots
U^{(-1)^{n-1}}=(\pm U)^n$$.

To verify (P4), fix arbitrary $U\in\mathcal N(e)$ and $x\in G$.
Choose a neighborhood $V\in\mathcal N(e)$ such that
$x^{-1}Vx\subset U$. Then $x^{-1}(\pm V)^nx=(\pm
x^{-1}Vx)^n\subset (\pm U)^n$.

To verify (P5), fix any $U\in\mathcal N(e)$. If $n$ is even, then
$((\pm U)^n)^{-1}=(\pm U)^n$. If $n$ is odd, use the assumption
$\tau_n\subset (\tau^{-1})_n$, to find $V\in\mathcal N(e)$ with
$(\mp V)^n\subset (\pm U)^n$. Then $((\pm V)^n)^{-1}= (\mp
V)^n\subset (\pm U)^n$. In any case the condition (P5) holds.

Hence the family $\mathcal B$ of $n$-oscillators in $G$ satisfies
the Pontriagin conditions (P1)--(P5) and since $\mathcal B$ forms
a neighborhood base of the topology $\tau_n$ at the unit of $G$,
$(G,\tau_n)$ is a topological group.
\end{proof}

Next, we consider some separation axioms for the oscillator
topologies. 
 We remind that a topology $\tau$ on a set $X$ is $T_1$
if for any distinct points $x,y\in X$ there is a neighborhood
$U\in\tau$ of $x$ such that $y\notin U$; $\tau $ is $T_2$ if the
topological space $(X,\tau)$ is Hausdorff.

\begin{theorem}\label{separ} For a paratopological group $(G,\tau)$ and
a positive integer $n$ the following conditions are equivalent:
\begin{enumerate}
\item the topology $\tau_n$ is $T_2$;
\item the topology $\tau_{2n}$ is $T_1$;
\item the topology $\tau_{2n+1}$ is $T_1$;
\item the topology $(\tau^{-1})_{2n+1}$ is $T_1$;
\item the topology $(\tau^{-1})_{2n}$ is $T_1$;
\item the topology $(\tau^{-1})_n$ is $T_2$;
\end{enumerate}
\end{theorem}

\begin{proof}  Let $\mathcal N(e)$ be a neighborhood base at the unit
$e$ of $G$ and $n$ be a positive integer.

$(1)\Rightarrow (2)$ Assume that the topology $\tau_n$ is
Hausdorff.
 This means that for any $x\ne e$ there is a neighborhood
$U\in \mathcal N(e)$ with $x(\pm U)^n\cap(\pm U)^n=\emptyset$.
Then $x\notin (\pm U)^n\cdot((\pm U)^n)^{-1}=(\pm U)^{2n}$ and
thus $\bigcap_{U\in\mathcal N(e)}(\pm U)^{2n}=\{e\}$, i.e., the
topology $\tau_{2n}$ is $T_1$.

$(2)\Rightarrow (3)$ Suppose the topology $\tau_{2n}$ is $T_1$.
Then $\bigcap_{U\in\mathcal N(e)}(\pm U)^{2n}=\{e\}$. To show that
the topology $\tau_{2n+1}$ is $T_1$ we have to verify that
$\bigcap_{U\in\mathcal N(e)}(\pm U)^{2n+1}=\{e\}$.

It suffices for each $x\ne e$ to find a neighborhood $W\in\mathcal
N(e)$ with $x\notin (\pm W)^{2n+1}$. Since the topology
$\tau_{2n}$ is $T_1$, there is $U\in\mathcal N(e)$ such that
$x\notin (\pm U)^{2n}$. Let $V,W\in\mathcal N(e)$ be such that
$V\cdot V\subset U$ and $W\subset V$, $Wx^{-1}\subset x^{-1}V$.
Then $xW^{-1}\subset V^{-1}x$. We claim that $x\notin (\pm
W)^{2n+1}$. Assuming the converse we would get $x\in (\pm
W)^{2n+1}=(\pm W)^{2n}W$ and consequently, $xW^{-1}\cap (\pm
W)^{2n}\ne\emptyset$. Since $xW^{-1}\subset V^{-1}x$, we get
$V^{-1}x\cap(\pm W)^{2n}\ne\emptyset$ and thus $x\in V(\pm
W)^{2n}\subset V(\pm V)^{2n}\subset (\pm U)^{2n}$, which
contradicts to the choice of the neighborhood $U$.

$(3)\Rightarrow (1)$ If the topology $\tau_{2n+1}$ is $T_1$, then
so is the topology $\tau_{2n}\supset\tau_{2n+1}$. Consequently,
for any distinct points $x,y\in G$ there is $U\in\mathcal N(e)$
such that $x^{-1}y\notin (\pm U)^{2n}=(\pm U)^n\cdot((\pm
U)^n)^{-1}$. Then $y\notin x(\pm U)^n\cdot((\pm U)^n)^{-1}$ and
consequently, $y(\pm U)^n\cap x(\pm U)^n=\emptyset$, i.e., the
topology $\tau_n$ is Hausdorff.

The equivalence $(3)\Leftrightarrow(4)$ follows from the equality
$(\tau_{2n+1})^{-1}=(\tau^{-1})_{2n+1}$.

Finally the equivalences $(4)\Leftrightarrow(5)\Leftrightarrow(6)$
follow from the equivalences
$(3)\Leftrightarrow(2)\Leftrightarrow(1)$ applied to the mirror
paratopological group $(G,\tau^{-1})$.
\end{proof}

Theorem~\ref{separ} allows us to introduce two number invariants
of paratopological groups reflecting their separatedness
properties.
 Given a paratopological
group $(G,\tau)$ and $i=1,2$ let $$T_i(G)=\sup\{n\in\IN:\mbox{ the
$n$-oscillator topology $\tau_n$ on $G$ is
$T_i$}\}\in\IN\cup\{\infty\}.$$ We assume that $\sup\emptyset=0$.
Thus a paratopological group $G$ is $T_i$ for $i=1,2$ if and only
if $T_i(G)>0$.

In terms of the invariants $T_1(G)$, $T_2(G)$, Theorem~\ref{separ}
can be reformulated as follows.

\begin{corollary}~\label{relation} If $G$ is a paratopological group
and $G^-$ is its mirror paratopological group, then
$T_1(G)=T_1(G^-)$, $T_2(G)=T_2(G^-)$, and $T_1(G)=2T_2(G)+1$. In
particular, $T_1(G)\ge 3$ for any Hausdorff paratopological group
$G$.
\end{corollary}

In \cite{Ra1} the second author has constructed a regular
zero-dimensional paratopological group $G$ with $T_1(G)=3$ and
$T_2(G)=1$. This shows that the lower estimation in the above
corollary cannot be improved. Below we use the idea of \cite{Ra1}
to construct a paratopological group $G$ with $T_2(G)=n$ for any
given $n\in\IN$.

Let $F$ be a free semigroup over a set $X$. A word $w=y_1\cdots
y_n\in F$, $y_i\in X$, is {\em reduced\/} if there is no pair
$y_iy_{i+1}$ such that $y_i^{-1}=y_{i+1}$. A reduced word is {\it
cyclic reduced} if $y_1^{-1}\not=y_n$.

\begin{lemma}\label{Le6.1}{\cite [Theorem 5.5]{LS}} Let $G$ be a group
generated by an alphabet $A=\{a,b,c,\dots\}$ with a relation
$r^p=1$ where $r$ is cyclic reduced and $p>1$. Let $w$ be a
nonempty reduced word in the alphabet $A$ such that $w$ is equal
to the unit of the group $G$. Then there exists a subword $s$ of
the word $w$ which also is a subword of the word $r^{p}$ or
$r^{-p}$ such that $l(s)>(p-1)l(r^p)/p$, where $l(s)$ and $l(r^p)$
denote the lengths of the words $s$ and $r^p$ respectively.
\end{lemma}

Under {\em the normal closure\/} of a subset $A$ of a group $G$ we
understand the smallest normal subgroup of $G$ containing the set
$A$.

\begin{lemma}\label{Co6.1} Let $F^2$ be a free group over a
two-point set $\{x,y\}$, $p>1$ be integer and $N$ be the normal
closure of the element $r^p=(xy^{-1})^p$ in $F^2$. Let $S\subset
G$ be a semigroup generated by the elements $x$ an $y$. Then $(\pm
S)^{2p-2}\cap N=\{e\}$.
\end{lemma}

\begin{proof} Let $w\in (\pm S)^{2p-2}\cap N$ be a non trivial reduced word.
Then Lemma \ref{Le6.1} implies that $w$ must contain a subword $s$
of length $>2p-2$ such that $s\not\in (\pm S)^{2p-2}$, which is
impossible.
\end{proof}

\begin{example}\label{ex0} For every positive integer $p$ there
exists a Lawson regular countable paratopological group $G$ with
$T_2(G)=p-1$ and $T_1(G)=2p-1$.
\end{example}

\begin{proof} Fix any positive integer $p$. For every $n\in\IN$,
let $F^2_n$ be a free group over a two-point set $\{x_n,y_n\}$.
Denote by $H=\oplus_{n=1}^\infty F^2_n$ the direct sum of the
groups $F^2_n$. Let $S_n\subset F^2_n$ be the semigroup generated
by the elements $x_n$ and $y_n$. Denote the direct sum
$\oplus_{m\ge n} S_m$ by $U_n$.

We show that the family $\{U_n:n\in\IN\}$ satisfies the Pontriagin
conditions (P1)--(P4) formulated in the proof of
Theorem~\ref{base}. The condition (P1) is satisfied because
$U_n\cap U_m\supset U_{\max(m,n)}$; (P2) and (P3) hold since $U_n$
are semigroups. To show that (P4) holds fix arbitrary $n$ and
$w\in G$. Find a number $m$ such that $w\in\oplus_{i=1}^m F^2_i$.
Then $w^{-1}U_{\max(m,n)+1}w= U_{\max(m,n)+1} \subset U_n$ and
thus the condition (P4) holds too. According to \cite{Ra1},
$\{U_n\}_{n\in\IN}$ is a neighborhood base at the unit of some
(not necessarily Hausdorff) paratopology on the group $H$.

Let $F_n$ denote the quotient group of the group $F^2_n$ by the
relation $r^p_n=(x_ny_n^{-1})^p$ and $\phi_n:F^2_n\to F_n$ be the
canonical homomorphism with $N_n=\ker\phi_n$. Let
$\psi_n:F^2_n\to\IZ$ be a unique homomorphism such that
$\psi_n(x_n)=1$ and $\psi_n(y_n)=0$. Define a map
$\psi:H\to\IZ\times\prod F_n$ as follows. Given $w=w_1\cdots
w_n\in H$ where $w_i\in F^2_i$ let
$\psi(w)=(\sum\psi_i(w_i),\prod\phi_i(w_i))$. Let $G=\psi(H)$ and
$\tau$ be the quotient paratopology on the group $G$, see
\cite{Ra1}. By definition, a base of this paratopology consists of
the sets $\psi(U_n)$, $n\in\IN$.

We claim that the $(p-1)$-oscillator topology $\tau_{p-1}$ on $G$
is Hausdorff. According to Theorem~\ref{separ}, it suffices to
show that the $(2p-2)$-oscillator topology $\tau_{2p-2}$ is $T_1$.
Observe that a neighborhood base of this topology consists of the
sets $(\pm \psi(U_n))^{2p-2}=\psi((\pm U_n)^{2p-2})$, $n\in\IN$.

To show that the topology $\tau_{2p-2}$ is $T_1$ it suffices given
an element $w\in H\backslash\ker\psi$ to find $n\in\IN$ with
$\psi(w)\notin\psi((\pm U_n)^{2p-2})$. Since
$H=\oplus_{i=1}^\infty F^2_i$, there is a positive integer $n$
with $w\in\oplus_{i=1}^nF^2_i$.

If $w\notin\oplus_{i=1}^n N_i$, then $\psi(w)\notin \psi((\pm
U_{n+1})^{2p-2})$. Next we consider the case $w\in\oplus_{i=1}^n
N_i$. We claim that $\psi(w)\notin \psi((\pm U_1)^{2p-2})$.
Assuming the converse we would find an element $s\in(\pm
U_1)^{2p-2}$ such that $\psi(s)=\psi(w)$. Lemma~\ref{Co6.1} yields
$s=e$ and hence $\psi(w)=\psi(s)=\psi(e)=e\not=\psi(w)$, which is
a contradiction. Hence the topology $\tau_{2p-2}$ is $T_1$ and the
topology $\tau_{p-1}$ is Hausdorff by Theorem~\ref{separ}.

Observe that the oscillator topology $\tau_{2p}$ is not $T_1$
since $(\psi(U_n)\psi(U_n)^{-1})^p\ni
(\psi(x_n)\psi(y_n)^{-1})^p=(p,e)$ for every natural $n$. It
follows that $T_1(G)=2p-1$ and $T_2(G)=p$.

Finally we show that $\psi(U_n)$ is a clopen subset of the group
$G$ for every $n$ and hence $\tau$ is a zero-dimensional regular
topology. Let $w\in H$ and $\psi(w)\in\overline{\psi(U_n)}$. Write
$w=w_1\cdots w_m$, where $m\ge n$ and $w_i\in F^2_i$ for $i\le m$.
There exist elements $u\in U_{m+1}, v\in U_n$ such that
$wuv^{-1}\in\ker\psi$. Write $u=u_{m+1}\cdots u_k$, $v=v_n\cdots
v_k$, where $u_i,v_i\in F^2_i$. Then $u_iv_i^{-1}\in N_i$ for
$i\ge m+1$. Since $u_iv_i^{-1}\in S_iS_i^{-1}$ for every $i$,
Lemma~\ref{Co6.1} implies that $u_i=v_i$ for $i\ge m+1$. Therefore
$w\prod\limits_{i=n}^m v_i^{-1}=wuv^{-1}\in\ker\psi$ and
$\psi(w)=\psi(v_n\cdots v_m)\in\psi(U_n)$
\end{proof}

Another number invariant of paratopological groups is suggested by
Theorem~\ref{base} which reflects some symmetry property of
paratopological groups, which will be referred to as {\em
oscillation symmetry}. We shall say that a paratopological group
$G$ has {\em finite oscillation\/} if there is a positive integer
$n$ such that any of the first four equivalent conditions of
Theorem~\ref{base} holds. In particular, $G$ has finite
oscillation if there is a positive integer $n$ such that for any
neighborhood $U$ of the unit in $G$ the set $(\pm U)^n$ is a
neighborhood of $e$ in $G^\flat$. We shall say that a
paratopological group $G$ has {\em countable oscillation\/} if for
any neighborhood $U\subset G$ of $e$ there is a positive integer
$n$ such that $(\pm U)^n$ is neighborhood of the unit in
$G^\flat$.

Next, we define an invariant of paratopological groups related to
the oscillation. This invariant takes value in the set $\mathbb
N\cup\{\omega,\infty\}$ linearly ordered so that
$n<m<\omega<\infty$ for each positive integers $n<m$. For a
paratopological group $G$ with finite oscillation let $\osc(G)$ be
the smallest positive integer $n$ such that for any neighborhood
$U\subset G$ of $e$ the set $(\pm U)^n$ is a neighborhood of $e$
in $G^\flat$. If $G$ has countable oscillation but fails to have
finite oscillation, then we put $\osc(G)=\omega$. If $G$ fails to
have countable oscillation we put $\osc(G)=\infty$. We shall say
that a paratopological group $G$ is {\em $n$-oscillating\/} if
$\osc(G)\le n$.

In particular, $\osc(G)\le 2$ (resp. $\osc(G)\le 3$) means that
the sets $UU^{-1}$ (resp. $UU^{-1}U$) with $U\in\mathcal N(e)$
form a neighborhood base at the unit of the topological group
$G^\flat$. The following Proposition is immediate.

\begin{proposition}\label{group} A paratopological group $G$ is a topological
group if and only if $\osc(G)=1$. \end{proposition}

Thus the oscillating number allows us to measure the distance from
a paratopological group to the class of topological groups, i.e.,
paratopological groups with small oscillation in a sense are near
to topological groups.

Next, we introduce a class of 2-oscillating paratopological groups
which contains all topological groups and all paratopological
SIN-groups. A paratopological group $G$ is defined to be a {\em
paratopological LSIN-group\/} if for any neighborhood $U$ of the
unit $e$ of $G$ there is a neighborhood $W\subset G$ of $e$ such
that $g^{-1}Wg\subset U$ for any $g\in W$. It is clear that each
topological group is a paratopological LSIN-group.

A paratopological group $G$ is {\em totally bounded\/} if for any
neighborhood $U$ of the unit $e$ of $G$ there is a finite subset
$F\subset G$ with $G=F\cdot U$. It is well-known that each totally
bounded topological group is a SIN-group. It is interesting to
remark that for paratopological groups it is not so, see
Example~\ref{ex2}.

\begin{proposition} Each paratopological SIN-group is a
paratopological LSIN-group. Conversely, each totally bounded
paratopological LSIN-group is a paratopological SIN-group.
\end{proposition}

\begin{proof} The first statement is trivial. To prove the second
statement, suppose that $G$ is a totally bounded paratopological
LSIN-group. Given a neighborhood $U$ of the unit $e$ in $G$, find
a neighborhood $V\subset G$ of $e$ such that $x^{-1}Vx\subset U$
for all $x\in V$. By the total boundedness of $G$ find a finite
subset $F\subset G$ such that $G=F\cdot V$. By the continuity of
the group operation, find a neighborhood $W\subset G$ of $e$ such
that $f^{-1}Wf\subset V$ for each $f\in F$. We claim that
$g^{-1}Wg\subset U$ for each $g\in G$. Indeed, given arbitrary
$g\in G$, find $f\in F$ and $x\in V$ such that $g=fx$. Then
$g^{-1}Wg=x^{-1}f^{-1}Wfx\subset x^{-1}Vx\subset U$.

Therefore, for any neighborhood $U\subset G$ of $e$ we have found
a neighborhood $W\subset G$ of $e$ such that $g^{-1}Wg\subset U$
for all $g\in G$. Hence $G$ is a
paratopological SIN-group.
\end{proof}

Following I.~Guran we say that a paratopological group $G$ is {\em
saturated\/} if for any neighborhood $U\subset G$ of the unit the
set $U^{-1}$ has nonempty interior in $G$. A standard example of a
saturated paratopological group with discontinuous inverse is the
{\em Sorgenfrey line\/}, i.e., the real line endowed with the
Sorgenfrey topology generated by the base consisting of
half-intervals $[a,b)$, $a<b$. Saturated paratopological groups
seem to be very close to being a topological group (this vague
thesis will be confirmed in the subsequent proposition). Let us
mention that totally bounded paratopological groups are saturated,
see Proposition 2.1 from \cite{Ra2}. The following theorem shows
that quite often we deal with $2$-oscillating paratopological
groups.

\begin{proposition}\label{sat} The class of paratopological 2-oscillating groups
contains all topological groups, all paratopological LSIN-groups
and all saturated paratopological groups.
\end{proposition}

\begin{proof} By Proposition~\ref{group}, the class of paratopological
2-oscillating groups contains all topological groups. To see that
each paratopological LSIN-group $G$ is 2-oscillating fix any
neighborhood $U$ of the unit $e$ in $G$ and find a neighborhood
$W\subset U$ of $e$ such that $g^{-1}Wg\subset U$ for all $g\in
W$. Then $g^{-1}W\subset Ug^{-1}$ for all $g\in W$ and thus
$W^{-1}W\subset UW^{-1}\subset UU^{-1}$. This means that the
paratopological group $G$ is 2-oscillating.

Finally, let us show that each saturated paratopological group
$(G,\tau)$ is 2-oscillating. Fix any neighborhood $U$ of the unit
$e$ in $G$. We have to find a neighborhood $W\subset G$ of $e$
such that $W^{-1}W\subset UU^{-1}$. Find an open neighborhood
$V\subset G$ of $e$ such that $V^2\subset U$. Since $G$ is
saturated, there are a point $x\in V$ and a neighborhood $W\subset
G$ of $e$ such that $x^{-1}W\subset V^{-1}$. Then $W^{-1}x\subset
V$ and $W^{-1}\subset Vx^{-1}$. We can assume that $W$ is so small
that $x^{-1}W\subset Vx^{-1}$. In this case $W^{-1}W\subset
Vx^{-1}W\subset VVx^{-1}\subset VVV^{-1}\subset UU^{-1}$. Hence
$\osc(G)\le 2$.
\end{proof}

Proposition~\ref{sat} gives topological conditions under which a
paratopological group is 2-oscillating. Next we consider some
algebraic conditions yielding the same result.

A group $G$ is defined to be {\em absolutely $n$-oscillating\/} if
any paratopological group algebraically isomorphic to $G$ is
$n$-oscillating. In particular, each abelian group is absolutely
$2$-oscillating and each group of finite exponent is absolutely
1-oscillating.

We shall show that the absolute $n$-oscillation property follows
from another algebraic property called the $n$-reversivity. A
group $G$ is defined to be {\em $(n,m)$-reversive\/} where
$n\in\IN$ and $m\in\IN\cup\{\infty\}$ if $(\mp A)^n\subset (\pm
A^m)^n$ for any subset $A\subset G$ containing the unit $e$ of $G$
(here $A^\infty=\bigcup_{n\in\IN}A^m\subset G$ is the semigroup in
$G$ generated by $A$). A group $G$ is called {\em $n$-reversive\/}
if it is $(n,m)$-reversive for some $m\in\IN$. Observe that each
$n$-reversive group is $(n,\infty)$-reversive and
$(n+1)$-reversive. Note also that a group $G$ is
$(1,\infty)$-reversive (resp. $1$-reversive) if and only if $G$ is
periodic ($G$ is of finite exponent).

Reversive groups were studied in \cite{Ba} where it was shown that
a group $G$ is 2-reversive if and only if $G$ is 3-reversive if
and only if $G$ is collapsing in the sense of \cite{SS}, \cite{S}.
We remind that a group $G$ is {\em collapsing\/} if there are
numbers $n,m\in\IN$ such that $|A^m|<|A|^m$ for any $n$-element
subset $A$ of $G$. Collapsing groups form a wide class of groups,
containing all groups with positive laws, in particular all
virtually nilpotent groups, see \cite{SS}, \cite{S},and \cite{Ma}.
We remind that a group $G$ is {\em virtually nilpotent\/} if $G$
contains a nilpotent subgroup of finite index. According to the
famous Gromov Theorem \cite{Gr}, a finitely generated group is
virtually nilpotent if and only if its has polynomial growth. For
finitely generated solvable groups a more precise characterization
is true: such a group is virtually nilpotent if and only if it
contains no free semigroup with two generators, see \cite{Ro}. It
is interesting to mention that a group $G$ contains no free
semigroup with two generators if and only if $G$ is
$(2,\infty)$-reversive if and only if $G$ is
$(3,\infty)$-reversive, see \cite{Ba}. Moreover, for any
polycyclic group $G$ the following conditions are equivalent: (i)
$G$ is virtually nilpotent, (ii) $G$ has polynomial growth, (iii)
$G$ is collapsing, (iv) $G$ contains no free semigroup with two
generators, (v) $G$ is $n$-reversive for some $n\in\IN$, (vi) $G$
is $(n,\infty)$-reversive for some $n\in\IN$, see \cite{Ba}.

For a group $G$ let $G^\omega_0=\{(g_n)\in G^\omega:g_n$ be the
unit of $G$ for almost all $n\}$ denote the direct sum of
countably many copies of $G$. The following proposition describes
the interplay between $n$-reversive and absolutely $n$-oscillating
groups.

\begin{proposition} \label{oscal} Let $G$ be a group and
$n$ be a positive integer.
\begin{enumerate}
\item If $G$ is $n$-reversive, then it is
absolutely $n$-oscillating.
\item If $G$ is not $n$-reversive, then
the group $G^\omega_0$ is not absolutely $n$-oscillating. More
precisely, $G^\omega_0$ is isomorphic to a first-countable
Hausdorff\/ $\flat$-regular zero-dimensional paratopological group
$H$ with $\osc(H)>n$.
\item If $G$ is isomorphic to $G^\omega_0$, then $G$ is
$n$-reversive if and only if $G$ is absolutely $n$-oscillating.
\end{enumerate}
\end{proposition}

\begin{proof} The last statement follows directly from the previous
two statements.

Assume that a group $G$ is $n$-reversive and find $m\in\IN$ such
that $(\mp A)^n\subset(\pm A^m)^n$ for each subset $A\subset G$
containing the unit $e$ of $G$. To show that $G$ is absolutely
$n$-oscillating, suppose that $\tau$ is a topology on $G$ making
the group operation of $G$ continuous. Given any neighborhood $U$
of the unit $e$ in $(G,\tau)$, find a neighborhood $W\subset G$ of
$e$ such that $W^m\subset U$. Then $(\mp W)^n\subset (\pm
W^m)^m\subset(\pm U)^n$. By Theorem~\ref{base} the paratopological
group $(G,\tau)$ is $n$-oscillating.

Next, assume that a group $G$ is not $n$-reversive. This means
that for any $m\in\IN$ there is a subset $A_m\subset G$ containing
the unit of $G_n$ such that $(\mp A_m)^n\not\subset(\pm A_m^m)^n$.

It is easy to find a countable family $\mathcal F$ of
non-decreasing maps $f:\omega\to\IZ_+=\{0\}\cup\IN$, satisfying
the following conditions \begin{enumerate}
\item $\lim_{n\to\infty}\frac{f(n)}n=0$ and
$\lim_{n\to\infty} f(n)=\infty$ for any $f\in\mathcal F$;
\item for any $f,g\in\mathcal F$ there is $h\in\mathcal F$ with
$2h\le \min\{f,g\}$.
\end{enumerate}

For any $f\in\mathcal F$ let $U_f=\{(g_m)_{m\in\omega}\in
G^\omega_0: g_m\in A_m^{f(m)}\}$. Repeating the argument from the
proof of Example~\ref{ex0} it can be shown that the family
$\mathcal B=\{U_f:f\in\mathcal F\}$ forms a neighborhood base at
the unit of some zero-dimensional first-countable paratopology
$\tau$ on $G^\omega_0$. Since each set $U_f$, $f\in\mathcal F$, is
closed in the Tychonov product topology on $G^\omega_0$ which is
weaker than $\tau$, we get that $U_f$ is $\flat$-closed. Hence the
paratopological group $(G^\omega_0,\tau)$ is Hausdorff and
$\flat$-regular. We claim that it not $n$-oscillating.

Assuming that $(G,\tau)$ is $n$-oscillating, we can find functions
$f,g\in\mathcal F$ such that $(\mp U_g)^n\subset (\pm U_f)^n$
which means that $(\mp A_m^{g(m)})^n\subset (\pm A_m^{f(m)})^n$
for all $m$. Find $m\in\omega$ such that $0<g(m)\le f(m)<m$. Then
$(\mp A_m)^n\subset (\mp A_m^{g(m)})\subset (\pm
A_m^{f(m)})^n\subset (\pm A_m^m)^n$ which contradicts to the
choice of the set $A_m$.
\end{proof}

A similar statement holds for $(n,\infty)$-reversive groups. We
remind that a paratopological group $G$ is {\em Lawson\/} if it
possesses a neighborhood base at the unit, consisting of
subsemigroups of $G$.

\begin{proposition} \label{oscal2} Let $G$ be a group and
$n$ be a positive integer.
\begin{enumerate}
\item If $G$ is $(n,\infty)$-reversive, then any
Lawson paratopology on $G$ is $n$-oscillating.
\item If $G$ is not $(n,\infty)$-reversive, then
the group $G^\omega_0$ is isomorphic to a Lawson first-countable
Hausdorff\/ $\flat$-regular zero-dimensional paratopological group
$H$ with\newline  $\osc(H)>n$.
\end{enumerate}
\end{proposition}

The first statement of this Proposition can be proven by analogy
with the proof of the first statement of Proposition~\ref{oscal}
while the second one follows from the next theorem whose proof
repeats the argument of Example~\ref{ex0} and
Proposition~\ref{oscal}.

\begin{proposition}\label{constr}
Suppose $G$ is a group and $S$ is a subsemigroup of $G$ containing
the unit $e$ of $G$. Then the sets $U_n=\{(g_n)_{n\in\omega}\in
G^\omega_0: g_i=e$ if $i\le n$ and $g_i\in S$ if $i>n\}$,
$n\in\IN$, form a neighborhood base of some Lawson paratopology
$\tau$ on $G$ which has the following properties:
\begin{enumerate}
\item the paratopological group $(G^\omega_0,\tau)$ is $\flat$-regular first-countable
and zero-dimensional;
\item the paratopological group $(G^\omega_0,\tau)$ is $n$-oscillating for
some $n\in\IN$ if and only if $(\mp S)^n\subset(\pm S)^n$.
\end{enumerate}
\end{proposition}

We use the above Proposition to construct an example of a
2-oscillating paratopological group whose mirror paratopological
group is not 2-oscillating. This example relies on a semigroup
with is left reversive but not right reversive. We remind that a
semigroup $S$ is {\em left\/} (resp. {\em right\/}) {\em
reversive\/} if for any elements $a,b\in S$ the intersection
$aS\cap bS$ (resp. $Sa\cap Sb$) is not empty, see
\cite[\S1.10]{CP}. If $S$ is a subsemigroup of a group, this is
equivalent to saying that $S^{-1}S\subset SS^{-1}$ (resp.
$SS^{-1}\subset S^{-1}S$). The simplest example of a semigroup
which is left reversive but not right reversive is the semigroup
$S$ generated by two transformations $y=2x$ and $y=x+1$ in the
group $\Aff(\IQ)$ of affine transformations of the field $\IQ$ of
rational numbers. This semigroup can be also defined in an
abstract way as a semigroup generated by two elements $a,b$ with
the relation $ab=b^2a$, see Example 1 to \cite[\S 1.10]{CP}. The
left and non-right reversivity of $S$ implies that $S^{-1}S\subset
SS^{-1}$ but $SS^{-1}\not\subset S^{-1}S$. Observe that the group
$\Aff(\IQ)$ is {\em metabelian\/} in the sense that it contains a
normal abelian subgroup with abelian quotient. Applying
Proposition~\ref{constr} to the semigroup
$S\cup\{e\}\subset\Aff(\IQ)$, we get the following unexpected
example showing that the oscillation number is powerful enough to
distinguish between a paratopological group and its mirror group.

\begin{example}\label{ex11} There is a Lawson Hausdorff countable first-countable metabelian
$\flat$-regular group $G$ with $\osc(G)=2$ and $\osc(G^-)=3$.
\end{example}

In spite of the fact that the oscillation numbers $\osc(G)$ and
$\osc^-(G)$ of a paratopological group $G$ and its mirror
paratopological group $G^-$ need not be equal, they cannot differ
very much. The following proposition can be easily derived from
the definitions and the equality $(\tau_n)^{-1}=(\tau^{-1})_n$
holding for each odd $n$.

\begin{proposition} Suppose $G$ is a topological group with
finite oscillation and $G^-$ is its mirror paratopological group.
Then \begin{enumerate}
\item $\osc(G)-1\le\osc(G^-)\le \osc(G)$ if the number $\osc(G)$ is odd;
\item $\osc(G)\le \osc(G^-)\le\osc(G)+1$ if $\osc(G)$ is even;
\end{enumerate}
\end{proposition}

It is clear that each 2-oscillating paratopological group is
3-oscillating. We shall show that (regular) 3-oscillating
paratopological groups are $\flat$-separated (and
$\flat$-regular).

\begin{theorem}\label{three} Any (regular) Hausdorff
3-oscillating paratopological group $G$
is $\flat$-separated (and $\flat$-regular).
\end{theorem}

\begin{proof} Suppose $G$ is a Hausdorff 3-oscillating paratopological group.
This means that the 3-oscillator topology $\tau_3$ coincides with
$\tau_\flat$. By Theorem~\ref{separ}, the topology $\tau_3$ is
$T_1$. Consequently, the topological group $G^\flat=(G,\tau_3)$ is
separated and hence is Hausdorff. This means that the group $G$ is
$\flat$-separated.

Next, let us verify that $G$ is $\flat$-regular provided $G$ is
regular. Fix any neighborhood $U$ of the unit $e$ of $G$. Since
$G$ is regular, we can assume that $U$ is closed in $G$. We have
to find a neighborhood $V\in\mathcal B$ such that the closure
$\overline{V}^\flat$ of $V$ in the topology $\tau_\flat$ lies in
$U$. Let $V$ be a neighborhood of $e$ in $G$ such that $V^3\subset
U$.

To show that $\overline{V}^\flat\subset U$, pick any point
$x\notin U$. We have to find a 3-oscillator $OO^{-1}O$ such that
$xOO^{-1}O\cap V=\emptyset$ (since $G$ is 3-oscillating
3-oscillators form a neighborhood base at the unit of $G^\flat$).
Since $U$ is closed in $G$, there is a neighborhood $W$ of $e$ in
$G$ such that $xW\cap U=\emptyset$. We can assume that $W$ is so
small that $Wx^{-1}\subset x^{-1}V$. The group $G$ is
3-oscillating and thus contains a neighborhood $O\subset G$ such
that $O^{-1}OO^{-1}\subset WW^{-1}W$. We claim that $xOO^{-1}O\cap
V=\emptyset$. Assuming the converse we would get $x\in
VO^{-1}OO^{-1}\subset VWW^{-1}W$ and thus $xW^{-1}\cap
VWW^{-1}\ne\emptyset$. Then $Wx^{-1}\cap
WW^{-1}V^{-1}\ne\emptyset$ and hence $x^{-1}V\cap
WW^{-1}V^{-1}\ne\emptyset$. After inversion we get $V^{-1}x\cap
VWW^{-1}\ne\emptyset$ and $x\in VVWW^{-1}\subset V^3W^{-1}$. Then
$xW\cap U\supset xW\cap V^3\ne\emptyset$ which contradicts to the
choice of the neighborhood $W$.
\end{proof}

Theorem~\ref{three} and Propositions~\ref{sat}, \ref{oscal} imply

\begin{corollary}\label{flat} A Hausdorff
(regular) paratopological group $G$ is $\flat$-separated
(and $\flat$-regular) provided $G$ satisfies one of the following
conditions:
\begin{enumerate}
\item $G$ is a saturated paratopological group;
\item $G$ is a paratopological LSIN-group;
\item $G$ is absolutely 2-oscillating;
\item $G$ is collapsing.
\end{enumerate}
\end{corollary}

Recall that a topological space $X$ is {\em \v Cech-complete\/} if
it is a $G_\delta$-set in any its compactification, see
\cite[\S3.9]{En}. It is well known that each complete metric space
is \v Cech-complete.

\begin{theorem} A Hausdorff $\flat$-regular paratopological group $(G,\tau)$
has countable oscillation provided the group reflexion $G^\flat$
is a Lindel\"of \v Cech-complete space.
\end{theorem}

\begin{proof} Let $\mathcal N(e)$ be a neighborhood base at the unit $e$
of the group $G$, consisting of $\flat$-closed sets. Then
$\mathcal B=\{\bigcup_{n\in\mathbb N}(\pm U)^n:U\in\mathcal
N(e)\}$ is a base at the unit of some (not necessary Hausdorff)
group topology on $G$ weaker than $\tau$. It follows that for any
neighborhood $U\in\mathcal N(e)$ the set $\bigcup_{n\in\mathbb
N}(\pm U)^n$, being an open subgroup of $G^\flat$, is closed in
$G^\flat$ and thus is Lindel\"of and \v Cech-complete.

Fix any neighborhood $U\in\mathcal N(e)$. We have to find $m$ such
that $(\pm U)^m$ is a neighborhood of $e$ in $G^\flat$. The group
$H=\bigcup_{n\in\IN}(\pm U)^n$, being \v Cech-complete, is Baire.
Consequently, there is $n\in\mathbb N$ such that the set $A=(\pm
U)^n$ is not meager in $H$. We claim that $A\cdot A^{-1}$ is a
neighborhood of the unit in $H$.

We shall use Banach-Kuratowski-Pettis Theorem (see
\cite[p.279]{Kel} or \cite[9.9]{Ke}) asserting that for any subset
$B$ of a topological group the set $BB^{-1}$ is a neighborhood of
the unit provided $B$ is non-meager and has the Baire Property in
the group. We remind that a subset $B$ of a topological space $X$
has {\em the Baire Property\/} in $X$ if $B$ contains a
$G_\delta$-subset $C$ of $X$ such that $B\setminus C$ is meager in
$X$. Thus to show that $AA^{-1}$ is a neighborhood of the unit in
$H$ it suffices to verify that the set $A=(\pm U)^n$ has the Baire
Property in $H$.

For this we shall use the well-known fact (see \cite{RJ} or
\cite[3.1]{Ha}) asserting that each $K$-analytic subspace $X$ of a
Tychonoff topological space $Y$ has the Baire Property in $Y$. We
remind that a topological space $X$ is {\em K-analytic\/} if $X$
is a continuous image of a Lindel\"of \v Cech-complete space. It
is known that the product of two K-analytic spaces is K-analytic
and the continuous image of a K-analytic space is K-analytic, see
\cite{RJ}. Observe that the subspace $U\subset H$, being a closed
subspace of the Lindel\"of \v Cech-complete space $H$, is
K-analytic. Then the space $A=(\pm U)^n\subset H$, being a
continuous image of the product $U^n$, is K-analytic too. Hence
$A$ has the Baire property in $H$ and by the
Banach-Kuratowski-Pettis Theorem, $AA^{-1}$ is a neighborhood of
the unit in $H$. Observing that $AA^{-1}\subset(\pm U)^{2n+2}$ we
see that the set $(\pm U)^m$ is a neighborhood of the unit of $H$
for $m=2n+2$. Since the group $H$ is open in $G^\flat$, we get
that $(\pm U)^m$ is a neighborhood of the unit in $G^\flat$.
\end{proof}

Next, we give a $\pi$-base characterization of saturated
paratopological groups. A collection $\mathcal W$ of non-empty
open subsets of a topological space $X$ is a {\em $\pi$-base\/} if
any non-empty open set $U\subset X$ contains an element $W$ of
$\mathcal W$.

\begin{theorem}\label{pibase} A paratopological group $G$ is
saturated if and only if the collection of nonempty $\flat$-open
subsets forms a $\pi$-base in $G$.
\end{theorem}

\begin{proof} To prove the  ``if'' part, assume that the collection
of nonempty $\flat$-open subsets forms a $\pi$-base for $G$ and
fix any neighborhood $U$ of the unit $e$ in $G$. Find a
$\flat$-open subset $V\subset U$. We can assume that $V=xW$ for
some $x\in V$ and some $\flat$-open neighborhood $W$ of $e$ with
$W=W^{-1}$. Then $W=W^{-1}$ is an open neighborhood of $e$ in $G$
such that $U\supset V=xW=xW^{-1}$ and thus $Wx^{-1}\subset U^{-1}$
which means that the paratopological group $G$ is saturated.

To prove the ``only if'' part, suppose that $G$ is a saturated
paratopological group. Fix any neighborhood $U\subset G$ of the
unit $e$ of $G$. Find a neighborhood $V\subset G$ of $e$ such that
$V\cdot V\subset U$. Since $V^{-1}$ has nonempty interior in $G$,
there is a point $x\in V$ and a neighborhood $W\subset V$ of the
unit $e$ such that $x^{-1}W\subset V^{-1}$. Then $W^{-1}x\subset
V$ and thus $WW^{-1}x\subset WV\subset U$. By
Proposition~\ref{sat}, the set $WW^{-1}$ is $\flat$-open. Hence
the collection of nonempty $\flat$-open subsets forms a $\pi$-base
for the space $G$.
\end{proof}

Theorem~\ref{pibase} implies that a saturated paratopological
group $G$ and its group reflexion $G^\flat$ have many common
properties (those that can be expressed via $\pi$-bases). In
particular, the spaces $G$ and $G^\flat$ have the same Souslin
number, the same calibers and precalibers, they simultaneously are
(or are not) Baire or quasicomplete and simultaneously satisfy (or
not) many chain conditions considered in \cite{CN} and \cite{AMN}
(such as the properties ccc, productively-ccc, $(*)$, $(**)$, (P),
or $(K_n)$ for $n\ge 2$), see also \cite{BR5}.

We present here only one result of this sort, concerning
precalibers of totally bounded paratopological groups. We remind
that a cardinal $\tau$ is a {\em precaliber\/} of a topological
space $X$ if any collection $\mathcal U$ of nonempty open subsets
of $X$ with $|\mathcal U|=\tau$ contains a centered subcollection
$\mathcal V$ with $|\mathcal V|=\tau$ (a collection $\mathcal V$
being {\em centered\/} if $\cap\mathcal F\ne\emptyset$ for any
finite subcollection $\mathcal F$ of $\mathcal V$). It is easy to
see that a topological space $X$ is countably cellular if
$\aleph_1$ is a precaliber of $X$ (the converse is true under
(MA$+\neg$CH) but is false under (CH), see \cite[p.43]{Ar} and
\cite{AMN}). It is well-known that each cardinal of uncountable
cofinality is a precaliber of any totally bounded topological
group (this follows from the dyadicity of compact topological
groups). This fact and Theorem~\ref{pibase} imply the following
useful result answering the ``paratopological'' version of
Protasov's Problem 6 from \cite{BCGP}.

\begin{corollary} A totally bounded paratopological group $G$ is
countably cellular. Moreover, each cardinal of uncountable
cofinality is a precaliber of $G$.
\end{corollary}

It is interesting to mention that for any infinite cardinal $\tau$
there is a zero-dimensional totally bounded left-topological group
with Souslin number $\tau$, see \cite{Pr}.
\smallskip

Theorem~\ref{three} implies that any paratopological group $G$
with $1\le T_2(G)<\infty$ satisfies $\osc(G)\ge 4$. In particular,
this concerns countable regular paratopological groups constructed
in Example~\ref{ex0}. Thus paratopological groups with large
oscillation numbers exist. Moreover, such a group can be a
subgroup of a paratopological group with small oscillation number.
(In this context it is interesting to notice that the class of
$\flat$-separated ($\flat$-regular) paratopological groups is
closed with respect to taking subgroups and many other
operations).

\begin{example}\label{ex2} There is a regular countable first-countable
saturated paratopological group $G$ with $\osc(G)=2$ containing a
$\flat$-closed subgroup $H$ with $\osc(G)=\infty$ and failing to
be a paratopological LSIN-group.
\end{example}

\begin{proof} We shall use the result of \cite{BR1} asserting that
a countable first-countable paratopological group $H$ is a
$\flat$-closed subgroup of a $\flat$-regular countable
first-countable saturated paratopological group provided $H$ has a
neighborhood base at the unit, consisting of subsets, closed in
some weaker topology $\sigma$ turning $H$ into a first-countable
topological SIN-group.

Thus to produce the required example it suffices to construct a
countable first-countable group $H$ with $\osc(H)=\infty$
possessing a neighborhood base at the unit, consisting of subsets
closed in some weaker topology $\sigma$ turning $H$ into a
first-countable topological SIN-group.

Consider the free group $F_2$ with two generators $x,y$ and the
unit $e$ and let $FS_2\subset F_2$ be the subsemigroup spanning
the set $\{e,x,y\}$. Let $H=(F_2)^\omega_0$ and $\tau$ be the
paratopology on $H$ generated by the semigroup $FS_2$ as indicated
in Proposition~\ref{constr}, which implies that the
paratopological group $(H,\tau)$ is countable,  first countable
and has infinite oscillation. Besides the topology $\tau$, the
group $H$ carries the weaker topology $\sigma$, induced from the
countable Tychonov power $(F_2)^\omega$ of the discrete group
$F_2$. It is easy to see that $(H,\sigma)$ is a topological
SIN-group and $(H,\tau)$ has a neighborhood base consisting of
$\sigma$-closed neighborhoods.

Applying \cite{BR1}, we conclude that $H$ is a $\flat$-closed
subgroup of a first-countable countable $\flat$-regular saturated
paratopological group $G$. The group $G$, being saturated, is
2-oscillating according to Proposition~\ref{sat}.

Assuming that $G$ is a paratopological LSIN-group, we would get
that so is its subgroup $H$ which is not possible because
$\osc(H)=\infty$, see Proposition~\ref{sat}.
\end{proof}

We saw in Example~\ref{ex11} that a paratopological group needs
not be isomorphic to its mirror paratopological group. Below we
construct a saturated example of this sort.

An automorphism $h:G\to G$ of a group is called an {\em inner
automorphism\/} of $G$ if there is $g\in G$ such that
$h(x)=g^{-1}xg$ for all $x\in G$.

\begin{proposition}\label{inner} Suppose $G$ is a paratopological group such that
any continuous automorphism $H:G^\flat\to G^\flat$ of its group
reflexion is an inner automorphism. The paratopological group $G$
is isomorphic to its mirror paratopological group $G^-$ if and
only if $G$ is a topological group.
\end{proposition}

\begin{proof} The ``if'' part of the theorem is trivial. To prove
the ``only if'' part, suppose that $h:G\to G^-$ is a topological
isomorphism. It follows that $h$ is a continuous automorphism of
the topological group $G^\flat$ and thus $h$ is an inner
automorphism. Find $g\in G$ with $h(x)=g^{-1}xg$ for all $x\in G$.
Then $x=gh(x)g^{-1}$ and hence the identity automorphism $id:G\to
G^-$ is continuous. This means that for any neighborhood $U\subset
G$ of the unit $e$ there is a neighborhood $V\subset G$ of $e$
such that $V\subset U^{-1}$, i.e., the inversion
$(\cdot)^{-1}:G\to G$ is continuous and hence $G$ is a topological
group.
\end{proof}

As usual, under a {\em character\/} on a topological group $G$ we
understand a continuous homomorphism $\chi :G\to\IT$ of $G$ into
the circle $\IT=\{z\in\IC:|z|=1\}$ considered as a multiplicative
subgroup of the complex plane $\IC$. Each character $\chi
:G\to\IT$ induces a topology (called {\em the Sorgenfrey
paratopology\/}) on $G$, whose neighborhood base at a point
$g_0\in G$ consists of the sets $U^+=\{g\in U: \Arg(\chi
(g_0))\le\Arg(\chi(g))<\Arg(\chi (g_0))+\pi\}$ where $U$ runs over
neighborhoods of $g_0$ in $G$ (as usual $\Arg(z)\in[0,2\pi)$
stands for the argument of a complex number $z\ne0$). It is easy
to see that $G$ endowed with the Sorgenfrey paratopology is a
saturated paratopological group. If the subgroup
$\Ker(\chi)=\chi^{-1}(1)$ is not open in $G$, then this
paratopological group fails to be a topological group. This
observation together with Proposition~\ref{inner} imply
\smallskip

\begin{corollary}~\label{char} Let $G$ be a topological group such
that each continuous automorphism of $G$ is inner and let $\chi
:G\to\IT$ be a character whose kernel $\Ker(\chi)$ is not open in
$G$. Suppose that $\tau$ is the Sorgenfrey paratopology on $G$
generated by the character $\chi $. Then the saturated
paratopological group $(G,\tau)$ is not isomorphic to its mirror
paratopological group $(G,\tau^{-1})$.
\end{corollary}

To construct a saturated paratopological group which is not
isomorphic to its mirror paratopological group, it rests to find
an example of a topological group satisfying the conditions of
Corollary~\ref{char}. Many such examples can be found using the
theory of Lie groups and Lie algebras, see \cite{GG}, \cite{VO}.

Probably the simplest example is the Lie group $\Aff^+(\IR)$ of
all orientation-preserving affine transformations of the real
line. This group can be represented by matrices of the form
$\left(\begin{smallmatrix} a&b\\0&1\end{smallmatrix}\right)$ where
$a,b\in\IR$, $a>0$. It is well-known that $\Aff^+(\IR)$ endowed
with the natural locally Euclidean topology is a metabelian Lie
group which is not a SIN-group (see \cite[p.279]{Kel}). It follows
from \cite[p.28]{GG} that any continuous automorphism of the Lie
group $\Aff^+(\IR)$ is inner. The group $\Aff^+(\IR)$ admits a
non-trivial character $\chi :\Aff^+(\IR)\to\IT$ assigning to each
matrix $A=\left(\begin{smallmatrix}
a&b\\0&1\end{smallmatrix}\right)\in\Aff^+(\IR)$ the complex number
$\chi (A)=e^{i\ln a}\in\IT$. This character induces the Sorgenfrey
topology $\tau$ on $\Aff^+(\IR)$ whose neighborhood base at the
unit $E=\left(\begin{smallmatrix}
1&0\\0&1\end{smallmatrix}\right)$ consists of the sets
$U(\varepsilon)=\left\{\left(\begin{smallmatrix}
a&b\\0&1\end{smallmatrix}\right): 1\le a<1+\varepsilon,\;
|b|<\varepsilon\right\}$ where $\varepsilon>0$. Thus it is legal
to apply Proposition~\ref{inner} and Corollary~\ref{char} to get

\begin{example}\label{ex3} The paratopological group $(\Aff^+(\IR),\tau)$
endowed with the Sorgenfrey topology $\tau$ is not isomorphic to
its mirror paratopological group $(\Aff^+(\IR),\tau^{-1})$. Yet,
the paratopological groups $(\Aff^+(\IR),\tau)$ and
$(\Aff^+(\IR),\tau^{-1})$ are saturated paratopological
LSIN-groups but are not paratopological SIN-groups.
\end{example}

\begin{remark} In fact many other Lie groups have the properties
of the group $\Aff^+(\IR)$. In particular, each non-trivial
solvable simply connected Lie group $G$, being a semidirect
product of a closed normal subgroup and a one-dimensional Lie
group, admits a non-trivial character, \cite[p.59]{VO}. If, in
addition, the Killing form of the Lie algebra of $G$ is
non-degenerated, then $G$ all automorphisms of $G$ are inner (see
\cite[\S1.5]{GG}) and thus $G$ admits a regular saturated
paratopology $\tau$ such that the paratopological group $(G,\tau)$
is not isomorphic to its mirror paratopological group
$(G,\tau^{-1})$.
\end{remark}

Finally, let us state some open questions related to the
introduced concepts.

\begin{problem}\begin{enumerate}
\item Is every $(2n+1)$-reversive group $2n$-reversive? {\rm (The
answer is ``yes' for $n=1$)}.
\item Is there an absolutely 2-oscillating group which is not
2-reversive?
\item Is every polycyclic group absolutely $n$-reversive for some
$n\in\IN$?
\item For which $n\in\IN$ there are a group $G$ and a subsemigroup
$S\subset G$ such that $(\mp S)^n\subset (\pm S)^n$ but $(\pm
S)^n\not\subset (\mp S)^n$?
\item Is every regular $\flat$-separated paratopological group
$\flat$-regular?
\item Suppose $G$ is a paratopological LSIN-group. Is the mirror
paratopological group $G^-$ a paratopological LSIN-group?
\item Is it true that for every number $n\ge 1$ there is a ($\flat$-regular)
paratopological group $G$ with $\osc(G)=n$? {\rm (The answer is
``yes'' for $n\le 3$)}.
\item Is there a (regular) Hausdorff paratopological group $G$
such that the numbers $T_2(G)$ and $\osc(G)$ are finite?
\item Is there a paratopological group $G$ whose all oscillator topologies
are Hausdorff, but the group reflexion $G^\flat$ of $G$ is not
separated?
\item Is there a paratopological group $(G,\tau)$ such that
$\tau_\flat\ne \inf \tau_n$?
\item Is there a paratopological group $G$ whose group reflexion
$G^\flat$ carries the antidiscrete topology?
\item Has a $\flat$-regular paratopological group $G$ finite
oscillation if its group reflexion $G^\flat$ is compact?
\end{enumerate}
\end{problem}

\noindent{\bf Acknowledgement.} The authors express their sincere
thanks to I.Yo.Guran and O.D.~Artemovych for valuable and
stimulating discussions on the subject of the paper.

\end{document}